\documentclass[10pt]{amsart}

\usepackage{amssymb,latexsym,palatino,mathrsfs} 
\usepackage{amscd}
\usepackage{hyperref}
\usepackage[all]{xy}
\usepackage{graphics}
\usepackage{float}

\input{sommers.tex}

\CompileMatrices

\newcommand{\vecdiff}[2]{e_{#1}\!-\!e_{#2}}

\newcommand{\g}{{\mathfrak g}}
\newcommand{\bo}{\mathfrak{b}} 
\newcommand{\nilrad}{\mathfrak{n}} 
\newcommand{\h}{\mathfrak{h}}

\newcommand{\complex}{\mathbf C}

\newcommand{\posroots}{\ro^+}

\newcommand{\ideal}{\mathcal I}
\newcommand{\idealJ}{\mathcal J}
\newcommand{\ideals}{\mathfrak Id}

\newcommand{\al}{{\alpha}}

\newcommand{\ro}{{\Phi}}

\newcommand{\orbit}{\mathcal O}

\usepackage[margin=1in,letterpaper,portrait]{geometry}

\title{A transitivity result for ad-nilpotent ideals in type $A$}

\author{Molly Fenn} \address{
N. C. State University\\
Raleigh, NC }
\email{mafenn2@ncsu.edu}

\author{Eric Sommers} \address{University of
Massachusetts---Amherst\\ Amherst, MA 01003}
\email{esommers@math.umass.edu}

\begin{document}

\begin{abstract}
The paper considers subspaces of the strictly upper triangular matrices, which are stable under 
Lie bracket with any upper triangular matrix.   These subspaces are called ad-nilpotent ideals and there are Catalan number of such subspaces.
Each ad-nilpotent ideal $I$ meets a unique largest nilpotent orbit $\orbit_I$ in the Lie algebra of all matrices.
The main result of the paper is that under an equivalence relation on ad-nilpotent ideals studied by Mizuno and others, 
the equivalence classes are the ad-nilpotent ideals $I$ such that $\orbit_I = \orbit$ for a fixed nilpotent orbit $\orbit$.
We include two applications of the result, one to the higher vanishing of cohomology groups of vector bundles on the flag variety 
and another to the Kazhdan-Lusztig cells in the affine Weyl group of the symmetric group.  Finally, some combinatorial results are discussed.
\end{abstract}

\maketitle

\section{Introduction}

Let $G$ be a connected, 
simple algebraic group over $\complex$ and $B$ a Borel
subgroup of $G$.  
Let $\g$ be the Lie algebra of $G$, $\bo$ the Lie
algebra of $B$, and $\nilrad$ the nilradical of $\bo$.
Fix a maximal torus $T$ in $B$ with Lie algebra $\h$.
Let $W$ be the Weyl group of $G$ relative to $T$.

The paper is concerned with the subspaces of $\nilrad$ which are stable under the adjoint action of $B$ 
(or the adjoint action of $\bo$).  They  are called ad-nilpotent ideals or $B$-stable ideals of $\nilrad$.  We will refer to them simply as ideals.  
Denote the set of all ideals by $\ideals$.    
Cellini and Papi \cite{cellini-papi:2} showed that the cardinality of $\ideals$ is the $W$-Catalan number of $G$, namely
$$\frac{1}{|W|}{\displaystyle  \prod_{i=1}^n{(h+d_i)}},$$
where $d_1, \dots, d_n$ are the fundamental degrees of $W$ and $h$ is the Coxeter number, which is the largest of the fundamental degrees; this reduces to the usual Catalan numbers in type $A$.
These ideals play a role in many structural results concerning nilpotent orbits in $\g$ and related objects in the representation theory of $G$, including 
 the partial order on nilpotent orbits under the closure relation \cite{gerstenhaber}, \cite{mizuno}, \cite{sommers:e_classes},
the representations of the corresponding group over a finite field \cite{kawanaka}, 
the vanishing of odd cohomology of Springer fibers \cite{dlp}, 
cells in the affine Weyl group \cite{shi:book}, Hessenberg varieties \cite{brosnan_chow}, \cite{harada_precup}, to name a few.   
They are also interesting from a purely combinatorial perspective since $\ideals$ is in bijection with Dyck paths of length $2n+2$ in type $A_n$,
as discussed in the last section.

Let $g.X:=Ad(g)(X)$ denote the adjoint action of $g\in G$ on $X \in \g$ and $G.X$ the orbit of $X$ under $G$. 
A nilpotent orbit in $\g$ refers to the orbit $\orbit$ of a nilpotent element.  There are finitely many such orbits.
Let $I\in\ideals$.   Then $I$ consists only of nilpotent elements.  
The $G$-saturation of $I$, denoted $G.I:=\{g.X \ | \ g \in G, X \in I  \}$, 
is the closure of a single nilpotent orbit, denoted $\orbit_I$, and $\orbit_I$ is also characterized 
as the unique nilpotent orbit such that $\orbit_I  \cap I$ is dense in $I$.   
We call $\orbit_I$ the {\it associated (nilpotent) orbit} of $I$.   Conversely, every nilpotent orbit arises as $\orbit_I$
for some $I$ by the Jacobson-Morozov theorem.  See, for example, \cite{sommers:e_classes} for these results.

The partial order on the set all nilpotent orbits is defined by containment of closures (that is, $\orbit_1 \leq \orbit_2$ if and only if $\overline \orbit_1 \subset \overline \orbit_2)$.  Clearly, if $I \subset J$, then $\orbit_I \leq \orbit_J$.   
Gerstenhaber computed the partial order for classical groups  \cite{gerstenhaber}.  
There, he introduced $\ideals$, which were called triangular subalgebras, and described an algorithm for finding $\orbit_I$ in type $A_n$.   
We review Gerstenhaber's algorithm and give a new proof in Section \ref{combinatorics}.
In \cite{mizuno} Mizuno described the partial order on nilpotent orbits in the exceptional Lie algebras and computed the component groups of nilpotent elements by studying many operations on $\ideals$ that preserve the associated nilpotent orbit.  This paper is concerned with the simplest operation (see Definition \ref{basicmove}),  which we call the {\it basic move}.  This operation is also important in studying cohomology of vector bundles on $G/B$ \cite{broer:vanishing}. 
 Our main result (see Theorem \ref{thmA}) is that in type $A_n$ this operation is transitive on the set of all $I \in \ideals$ with the same associated orbit.   A transitivity result of this kind is proved in \cite{dlp} for some of the ideals in the exceptional groups when $\orbit_I$ is distinguished. 
We give two applications of Theorem \ref{thmA} in Section \ref{applications} and describe some combinatorially results in Section \ref{combinatorics}.  

In other types, additional operations on $\ideals$ were studied in the PhD thesis of the first author \cite{fenn}, where it is conjectured that an analogous  transitivity result holds when these additional operations are allowed.   This has been checked in the exceptional groups by computer, but is still open for other classical types.  Indeed, if one knows such transitivity results, then the partial order on nilpotent orbits reduces to a combinatorial problem involving only $\ideals$.
Namely, consider equivalence classes $[I]$ on $\ideals$ defined by the condition $\orbit_I = \orbit$.  
The partial order on $\ideals$ by inclusion of ideals induces a partial order on equivalence classes in $\ideals$.  
Then, $[I] \leq [J]$ if and only if $\orbit_I \leq \orbit_J$.  This was proved uniformly in \cite{sommers:e_classes}.  Of course, the idea is exactly the one used by Gerstenhaber and Mizuno to explicitly compute the partial order.

\section{Basic move}

Let $\ro  \subset \h^*$ be the roots of $G$ on $\g$ relative to $T$.  
Let $\ro^+$ and $\Pi$ be the positive roots and simple roots determined by $B$.  
For $\beta \in \ro$, let $\g_\beta$ be the $\beta$ weight space in $\g$.

Since ideals are $T$-stable, any $I \in \ideals$ decomposes as sum of root spaces.  Define $\ideal \subset \ro^+$ 
by the equation 
\begin{eqnarray}\label{ideal_to_cominbatorial_ideal}
I = \displaystyle \bigoplus_{\beta \in \ideal} \g_\beta.
\end{eqnarray}

Let $\al \prec \beta$ denote the usual partial order on roots; that is,
$\al \prec \beta$ if and only if $\beta - \al$ is a sum of positive roots.
Then $\ideal$ has the property that if $\al \in \ideal$
and $\al \prec \beta$, then $\beta \in \ideal$.
In turn, any $\ideal \subset \posroots$ with this property determines
a unique $I \in \ideals$ via equation \eqref{ideal_to_cominbatorial_ideal}.
Such an $\ideal$ is an upper order ideal in the poset $\ro^+$ and we refer 
refer to them as root  ideals or just ideals and go back and forth between $I$ and its root ideal $\ideal$.   
The set of minimal roots 
in $\ideal$ form an antichain (no two distinct elements are comparable) and each 
$\ideal$ is determined  by its antichain.
Let $\ideal_{min}$ denote the minimal roots of $\ideal$.

For $\al \in \Pi$, let $P_\al$ be the parabolic subgroup of $G$ containing $B$ 
attached to $\al$ and let $s_\al \in W$ be
the simple reflection attached to $\al$.  Then $P_\al/B$ is a projective line.
Given distinct $I, J \in \ideals$, suppose there exists $\al \in \Pi$ such that $P_\al. J= I$.  
Then necessarily $J \subset I$ and (1) $\dim I = \dim J +1$ since $P_\al/B$ is dimension one; (2) $I$ is stable under $P_\al$; and 
(3) $\orbit_I = \orbit_J$ since $\orbit_I$ meets $J$ and $\orbit_J \leq \orbit_I$.   
For such a pair of ideals to exist, we must have 
$I = J \oplus \g_{\beta}$ with $\beta$ a minimal root of the associated root ideal $\ideal$, and $s_\al(\ideal) = \ideal$.
Hence, $s_\al(\beta) = \beta+k\alpha \in \ideal$ and thus $k<0$ since $k=0$ contradicts $P_\al.J \neq J$ and $k>0$ 
contradicts that $\beta \in \ideal_{min}$.  
Thus $\langle \beta, \al^\vee \rangle \in \{-1,-2,-3\}$. For some applications, it is useful to insist that this value is $-1$ (of course, this is always true in type $A$ and other simply-laced types).   

\begin{defn}  
\label{basicmove}	
	Let $I \in \ideals$ and $\ideal$ the associated root ideal.  Let $\al \in \Pi$ and $\beta \in \ideal_{min}$ with 
$\langle \beta, \al^\vee \rangle = -1$.   Then $\idealJ := \ideal \, \backslash \, \{\beta \}$ is a root ideal (with associated ideal $J \in \ideals$).
If $s_\al(\ideal) = \ideal$, then $\orbit_J = \orbit_I$ and we say that $I$ and $J$ (or $\ideal$ and $\idealJ$) are related by the {\it basic move} and write $I \sim J$ (respectively, $\ideal \sim \idealJ$).
\end{defn}

We can now state the main result.
\begin{thm}  \label{thmA}
	Let $I,J \in \ideals$.  	In type $A_n$ if $\orbit_I = \orbit_J$, then $I$ and $J$ are related by a sequence of basic moves.  
	That is, the equivalence relation on $\ideals$ induced by the transitive action of basic moves coincides with the equivalence relation 
	on $\ideal$ defined by having the same associated nilpotent orbit.
\end{thm}

\section{Preliminaries}

From now on, let $G$ be the general linear group $GL_{n+1}(\complex)$ and 
$B$ the upper triangular matrices and $T$ the diagonal matrices in $G$.
Let $t_{ij}$ be the matrix with $1$ in the $(i,j)$ spot and zero elsewhere.  
Take the standard basis $\{e_i\}$ of $\h^*$ so that the simple roots are 
$\al_i = e_i-e_{i+1}$ for $1 \leq i \leq n$.  Each positive root has the form 
$\alpha_i +\alpha_{i+1} + \cdots + \alpha_j = \vecdiff{i}{j+1}$ for $1 \leq i<j \leq n$, 
corresponding to the weight space in $\g =M_{n+1}(\complex)$ spanned by $t_{ij}$, 
and we denote this root by $[ i , j ]$.
The usual partial order on the positive roots is expressed as $[ i ,
j ] \preceq [i', j']$ if and only if $i' \le i$ and $j \le j'$.  
Since a root ideal $\ideal$ is determined by the minimal roots in 
$\ideal$ under the partial order, 
we can specify an ideal by its collection of minimal roots and write 
 $\ideal_{min}= \{  [a_k, b_k]\}$ for the minimal roots.
 If $[a,b] \in \ideal_{min}$ we say $a$ is a left endpoint of $\ideal$ and $b$ is right endpoint of $\ideal$.
If $s_{\alpha_{j}}(\ideal) =\ideal$, we say $\ideal$ is $j$-stable.
This translates to saying $j$ is neither a right endpoint, nor a left endpoint, of $\ideal$.

Let $P$ be a standard parabolic subgroup of $G$, i.e., a closed subgroup containing $B$.  
Let $\nilrad_P$ denote the nilradical of the Lie algebra of $P$.  Then $\nilrad_P \in \ideals$.  
We will also denote by $\nilrad_P$ the associated root ideal.  
All of the minimal roots of $\nilrad_P$ are simple roots, and conversely if 
$\ideal_{min}$ are all simple, 
then $I$ is equal to some $\nilrad_P$. 
We can designate $P:=P_{j_1, \dots, j_l}$ in terms of the minimal roots 
$\al_{j_1}, \dots, \al_{j_l}$ of $\nilrad_P$ where $1 \leq j_1 <   \dots < j_l \leq n$.   
The sequence $(j_1, \dots, j_l)$ gives rise to a composition of $n+1$ with $l+1$ parts:
\begin{equation}\label{composition}
(c_1, \dots, c_l, c_{l+1}):=(j_1, j_2-j_1, \dots, j_l-j_{l-1}, (n+1)-j_l).
\end{equation}
The standard Levi subgroup of $P_{j_1, \dots, j_l}$ is isomorphic 
to $\prod_i GL_{c_i}$, embedded as 
block diagonal matrices in $G$.
Denote by $\mu_P$ the partition of $n+1$ obtained by arranging the composition in \eqref{composition} in weakly descending order.

The proof of Theorem \ref{thmA} proceeds by showing that every $I \in \ideals$ is equivalent to some $\nilrad_P$.  Then it is shown that $\nilrad_P \sim \nilrad_{P'}$ if
and only if $\mu_P = \mu_{P'}$.  

We start with some lemmas.
Let $[a,b] \in  \ideal_{min}$.   Omitting the root $[a,b]$ from $\ideal$ gives a new 
ideal $\idealJ$.   
Let $A = \ideal_{min} \, \backslash \{[a,b]\}$.   Then $\idealJ_{min} = A \cup S$
where $S \subset \{[a-1, b], [a, b+1] \}$ specified by 
$[a-1, b] \in S$ if $a-1$ is not $0$ nor a left endpoint of $\ideal_{min}$
and $[a, b+1] \in S$ if $b+1$ is not $n+1$ nor a right endpoint of $\ideal_{min}$.
We now describe the basic move from Definition \ref{basicmove} in type $A_n$. 

\begin{defn}[Basic Move in type $A_n$]
	\label{Abasicmove}
Let $\ideal$ be a root ideal and $[a,b] \in \ideal_{min}$. 
If $\ideal$ is $(a\!-\!1)$-stable or $(b\!+\!1)$-stable.  Then $\ideal$ and $\idealJ:=\ideal \backslash \{[a,b]\}$ are equivalent under the basic move (\ref{basicmove}) and we  write $\ideal \sim \idealJ$ or $\ideal_{min}\sim \idealJ_{min}$.
\end{defn}

We record a few lemmas that follow quickly from the basic move.

\begin{lem}\label{stacked_move}
Let $\ideal$ be an ideal with $\ideal_{min} = A \cup \{[a,j], [b, j+1]\}$ for $a<b \leq j$.
Then $[b,j]$ is not comparable to any root in $A$ so there is an ideal $\idealJ$
with $\idealJ_{min} = A \cup \{[b,j]\}$.   
Assume the intervals in $A$ contain no endpoints $j$ with $a<j<b$.  
Then $\ideal \sim \idealJ$.  
\end{lem}

\begin{proof}
The hypothesis on $A$ means that $\idealJ$ is $(b-1)$-stable, so 
$$\idealJ_{min} \sim  A \cup\{[b-1, j],[b,j+1]\}.$$	 Set $A' = A \cup\{[b,j+1]\}$.
Then using the basic move multiple times gives
$$A \cup\{[b-1, j],[b,j+1]\} \sim A' \cup\{[b-2, j]\} 
\sim \dots \sim A'\cup\{[a, j]\} $$
since $A'$ does not contain $b-2, b-3, \dots, a+1$.
\end{proof}

A similar argument gives
\begin{lem}\label{endpoint_move}
	Let $\ideal \in \ideals$ with $\ideal_{min} = A \cup \{[a,n]\}$.
	Let $\idealJ$ have $\idealJ_{min} = A \cup \{[b,n]\}$.   
	Assume the intervals in $A$ contain no endpoints $j$ with $a<j<b$.  
	Then $\ideal \sim \idealJ$.  	The result also holds if 
	$[a,n]$ is replaced by $[1,b]$ and $[b,n]$ by $[1,a]$ 
	in the definitions of $\ideal$ and $\idealJ$.
\end{lem}


\begin{exam}
In $A_3$, we have
$\{[1]\} \sim \{[1,2]\} \sim \{[1,3]\}$ and 
$\{[3]\} \sim \{[2,3]\}  \sim \{[1,3]\}$ by Lemma \ref{endpoint_move} with $A = \emptyset$.  
And $\{[2]\} \sim \{[1,2],[2,3]\}$ by Lemma \ref{stacked_move} with $a=1, b=2, j=2$ and $A = \emptyset$.
Finally,  $\{[1],[2]\} \sim \{[1], [2,3]\} \sim \{[1], [3]\} \sim \{[1,2], [3]\} \sim \{[2], [3] \}$ via basic moves.
Each of $\emptyset$ and $\{[1],[2],[3]\}$ are their own equivalence class and this accounts for the $5$ classes among the $14$ elements of $\ideals$.
Also each class contains at least one $\nilrad_{P}$.
\end{exam}

%
%
%
%
%

%

\section{Proof of Theorem \ref{thmA}}

We call an antichain $S \subset \ro^+$ a {\it right staircase} if 
there exist positive integers $c \leq d$ and $a_i$'s such that
\begin{equation}
	S = \{[a_i, i] \ | \ c \leq i \leq d \} \text{ and either } d =n \text{ or } c \leq a_d.
\end{equation}
Note that $a_i < a_j$ for $c \leq i <j \leq d$.
Examples of such $S$ in $A_8$ are 
$\{ [2,5],[3,6],[6,7] \}$ and $\{ [1,5],[3,6],[4,8] \}$. 

\begin{lem} \label{no_lefties}
If $\ideal_{min}$ 
contains a right staircase, then no minimal root of $\ideal$ 
can have left endpoint $j$ with $a_c \leq j \leq c$
unless $j = a_i$ for some $i$. 
\end{lem}

\begin{proof}
If $[j,k]$ were a minimal root with $a_c \leq j \leq c$, 
then $k > c$; otherwise $[j,k] \prec [a_c,c]$.  
Furthermore $k>d$ since only one minimal root can have a given right endpoint and the values $i$
with $c \leq i \leq d$ are spoken for.   But if $k>d$, then $d \neq n$ and we must have $c \leq a_d$.
But then $[a_d, d] \prec [j,k]$ since $j \leq c \leq a_d$ and $k>d$, which is a contradiction.
\end{proof}

\begin{lem}\label{staircase_exists}
	Every ideal is equivalent to an ideal containing a right staircase. 
\end{lem}
\begin{proof}
	Let $[a,b]$ be the interval of $\ideal$ with largest $b$. 
	If $b<n$, 	Then $\ideal$ is $b+1$-stable, so 
	apply the basic move and omit $[a,b]$ 
	to obtain an equivalent ideal with 
	interval $[a,b+1]$.  Repeating this process, we see that 
	the equivalence class of $\ideal$ 
	contains an ideal with interval $[a,n]$, which qualifies as a right staircase.
\end{proof}

If $\ideal_{min}$ contains a right staircase $S$ and furthermore 
no minimal root of $\ideal$ has right endpoint $j$ with $a_c \leq j \leq c-1$,
then we say $S$ is a {\it pure right staircase} in $\ideal_{min}$.

\begin{lem}\label{pure_staircase}
	Suppose $\ideal_{min}$  contains the pure right staircase
	$$S:= \{[a_i, i] \ | \ c \leq i \leq d \}.$$
		Then $\ideal$ is equivalent to an ideal containing the simple root $\alpha_c = [c,c]=[c]$. 
\end{lem}

\begin{proof}
	We can assume $a_c <c$ , otherwise $[c] \in S$ and there is nothing to show.
	
	Suppose $a_d <c$.  We will show that $\ideal$ is equivalent to an ideal containing the same staircase except that
	$c \leq a_d$.  Since $a_d<c$, then $d=n$ from the definition of a right staircase.  
	For any $j$ with $a_d <j<c$, $j$ cannot be  a right endpoint by the pure hypothesis and the fact that $a_c \leq a_d$; 
	nor can $j$ be a left endpoint by Lemma \ref{no_lefties}.
	This means $[a_d, n]$ can be replaced in $\ideal$ by $[c,n]$ to yield an equivalent ideal 
	by Lemma \ref{endpoint_move}, completing this part of the argument.
	
  	Next, we show that if $c <a_d$, then $\ideal$ is equivalent to an ideal containing a pure staircase with $a_d =c$.
	Let $j$ be the largest number with the property that $a_j< c$.  Such a $j$ exists since we assumed $a_c < c$.   
	Then $a_c \leq a_j \leq c-1$.  By the pure assumption $a_j$	
	does not occur as the right endpoint of a minimal root of $\ideal$.
	Since $c < a_d$, $j \neq d$ and thus $[a_{j+1}, j+1]$ is in the staircase.
	Now we can apply (the second part of the proof of) Lemma \ref{stacked_move} 
	to replace $[a_j,j]$ with $[c, j]$ since $k$ is neither a right nor left endpoint in $\ideal$ for
	$a_j < k < c$.   This shows that $\ideal$ is equivalent to an ideal containing a pure staircase with $a_d = c$ since we can just 
	forget all the later roots in the original staircase.
	
	If $c=d$, we are already done.  If not, we can apply Lemma \ref{stacked_move} to 
	$[a_{d-1}, d-1]$ and $[a_d=c, d]$, replacing those 
	with $[c, d-1]$.  The proof follows by induction on the difference $d-c$ and 
	we arrive at the interval $[c]$, as desired.
\end{proof}

\begin{exam}
The steps in the lemma for $\{ [2,5],[3,6],[6,7] \}$  are with $c=5$:
$$\{ [2,5],[3,6],[6,7] \} \sim  \{ [2,5],[5,6],[6,7] \}$$  and then
$$ \{ [2,5],[5,6, [6,7]]\} \sim  \{ [5], [6,7] \}.$$ 
\end{exam}

\begin{cor}\label{grassmann_switch}
	For any $k \geq \frac{n-1}{2}$, the root ideal 
	$$\ideal:= \{[j, j+k] \ | \ 1 \leq j \leq n-k \}$$
	is equivalent to $\{[k+1]\}$ and also to $\{[n-k]\}$.
\end{cor}
\begin{proof}
	This is a pure staircase with $c = k+1$ with $a_d = n-k$.  
	Since $n-k \leq k+1$, we have $a_d\leq c$, so the first part of the previous proof  implies that $\ideal$ 
	is equivalent to a pure staircase with $a_d=c$.   Then the last paragraph implies $\ideal \sim \{[c]\}$.   	A symmetric argument as in the lemma but moving right endpoints instead of left gives  $\ideal \sim \{[n-k]\}$. 
\end{proof}
	
\begin{prop} \label{has_a_simple}
	Every ideal $\ideal$ is equivalent to an ideal with a minimal root which is a simple root.
	 Hence, every ideal is equivalent to some $\nilrad_P$.
\end{prop}

\begin{proof}
	By Lemma \ref{staircase_exists}, we can assume that $\ideal$ contains a right staircase.
 	Next, we prove that if $\ideal$ contains a 
 	right staircase $S=\{[a_i, i] \ | \ c \leq i \leq d \}$
 	that is not pure, then $\ideal$  is equivalent 
 	to an ideal with a right staircase with smaller value of $c$.
 	Then the proposition will follow:  either we encounter a right staricase that is pure and we can apply Lemma \ref{pure_staircase}, or we
 	eventually arrive at a staircase with $c=1$, which means $a_c=c=1$; that is, $[1]\in \ideal_{min}$.
	
	Therefore, assume that $S$ is not pure.  Then there exists a largest $m<c$
	with $a_c \leq m$ so that $[l,m] \in \ideal_{min}$. 
 	Define	$L = \{j \ | \ m < a_j < c\}$.
 	If $L$ is not empty, let $j_0$ be its smallest element.  
 	Then we can apply 
 	the argument from Lemma \ref{pure_staircase} to the right staircase
 	$S'=\{[a_i, i] \ | \ j_0 \leq i \leq d \}$ to replace it with $[c,j_0]$
 	arriving at an equivalent ideal 
  	containing	$[a_m, m]$ and no endpoints strictly between 
 	$m$ and $c$.   This statement is also trivially true if $L$ is empty.
 	
 	Now, if $m=c-1$, we are done; otherwise $\ideal$ is stable under $m+1$.
 	And so we can replace $[a_m,m]$
 	by an equivalent ideal containing $[a_m,m+1]$ which has no endpoints 
 	strictly between $m+1$ and $c$.   It follows by induction on 
 	$c-m$ that $\ideal$ is equivalent to an ideal with a right staircase
 	that begins with $[a_m,c-1]$ and ends with either $[c,j_0]$ ($L$ not empty]) 
 	or $[a_d, d]$ ($L$ empty).   This completes the proof of the first statement.
 	
 	Now, given $I \in \ideal$, it is equivalent to an ideal
 	with minimal simple root $\al_i$ for some $i$.
 	The other minimal roots for $\ideal$ must all
 	come from one of the two irreducible
 	root subsystems generated by the remaining simple roots,
 	which are of type $A_{i-1} $ and $A_{n-i}$.
 	By induction on $n$ we can assume the theorem in 
 	types $A_{i-1}$ and $A_{n-i}$, and hence the ideal $\ideal$
 	is equivalent to one whose minimal roots are all simple.
\end{proof}



\begin{lem}\label{associated_parabolics}
	For standard parabolic $P, P'$, we have $\nilrad_P \sim \nilrad_P'$ 
	if and only if  $\mu_P = \mu_P'$. 
\end{lem}

\begin{proof}
	Let the minimal roots of $\nilrad_P$ be $[j_1], \dots ,[j_l]$.   Set $j_0=0$ and $j_{l+1} = n+1$. 
	Let $1 \leq b \leq l$.
	Then Corollary \ref{grassmann_switch}  applies to each subsystem 
	with simple roots $\al_i$ for $j_{b-1}+1 \leq i \leq j_{b+1}-1$.  
	Namely, It implies that $\nilrad_P$ is  equivalent to $\nilrad_{P'}$ 
	with the same minimal roots except that $[j_{b}]$ is replaced with $[j_{b-1} + j_{b+1} -j_{b}]$.   
The composition for $\nilrad_P'$ is the same as the one for 
$\nilrad_{P}$ except that the $j_{b+1}-j_b$ and $j_{b}-j_{b-1}$ exchange (adjacent) positions.
	Thus this action acts like the simple transposition $(b, b+1)$ on the $l+1$ elements of the composition, generating the symmetric $S_{l+1}$.  This shows all rearrangements of the composition yield equivalent ideals and therefore that  if $\mu_P = \mu_P'$, then $\nilrad_P \sim \nilrad_{P'}$.   Together with the Proposition \ref{has_a_simple}, we know that the number of equivalence classes is at most $n+1$.
	
	On the other hand, the number of nilpotent orbits equals the number of partitions of $n+1$ and every orbit arises as $\orbit_I$ for some $I \in \ideals$ (see the introduction).  Hence the number of equivalence classes is at least the number of partitions of $n+1$.   We conclude the number of classes is exactly the number of partitions of $n+1$.  Hence the converse statement is also true: distinct partitions correspond to distinct equivalence classes. (We will recall in 
	Section \ref{algorithm} that the Jordan type of $\orbit_{\nilrad_P}$ is the dual partition of $\mu_P$).
\end{proof}

The proof of Theorem \ref{thmA} is now complete by the proof of Lemma \ref{associated_parabolics}:  the equivalence classes under the basic move coincide with the equivalence classes under associated nilpotent orbit.

\section{Two applications}
\label{applications}

\subsection{Vanishing cohomology}

For a rational representation $V$ of $B$, denote by 
$H^*(G/B, V)$, the cohomology of the sheaf of sections of the vector bundle
$G \times^B V$ over $G/B$.   Let $V^*$ denote the linear dual of $V$ and
$S^j V^*$ the $j$-th symmetric power of $V^*$, which are all $B$-modules.

For each parabolic subgroup $P$ containing $B$, 
it is known in all Lie types that $H^i(G/B, S^j \nilrad_P^*)=0$ 
for all $j \geq 0$ and $i>0$ (see \cite[Chapter 8]{jantzen}).
By Proposition 4.3 in \cite{achar-sommers:local}, if $I$ and $J$ are related by the basic move, 
then 
\begin{equation}
H^i(G/B, S^jI^*)= H^i(G/B, S^j J^*) \text{ for all } i, j \geq 0.
\end{equation}
Hence a corollary of Theorem \ref{thmA}
is the following.
\begin{cor} \label{higher_vanishing}
	In type $A_n$, every $B$-stable ideal $I$ of $\nilrad$ 
	satisfies
	$H^i(G/B, S^j I^*)=0$ for all $j \geq 0$ and $i>0$.
\end{cor}

We expect this to hold in all types and it has been checked for most ideals in the exceptional groups \cite{fenn}.
The vanishing result is already known in all types for the ideal $\g_{\geq 2}$ 
attached to each orbit via the Jacobson-Morozov theorem \cite{hinich:van}, \cite{panyushev:van}, so 
the corollary gives a new proof for $\g_{\geq 2}$ in type $A$.
The vanishing also implies formulas for the $G$-module of graded functions on $\orbit_I$.  
Let $\ideal^c := \ro^+ \backslash \ideal$ be the lower order ideal.  Then 
for any $I$ in type $A$, the ungraded functions on $\orbit_I$ satisfy
$$R(\orbit_I) = \Ind_T^G \left( \prod_{\alpha \in \ideal^c} (e^0-e^\alpha) \right)$$
in the notation of Corollary 3.2 in \cite{mcgovern}.  A graded formula is obtained by
the Kostant multiplicity formula using Lusztig's q-analog of weight multiplicity, but 
where the positive roots allowed are drawn only from $\ideal$ (see \cite[Chapter 8]{jantzen}).

\subsection{Cells in the affine Weyl group}

Shi \cite{shi:book} computed the left cells in the affine Weyl group
of type $A_n$. 
The left cells, viewed geometrically, are unions of regions of the Shi arrangement.  
In the dominant chamber, the regions of the Shi arrangement are indexed by positive sign types
and these are in bijection with $\ideals$.   For $I \in \ideals$, let $R_I$ be the associated region.
Then Shi has two results: (1) for $I\in \ideals$, the region $R_I$ lies in a left cell;  (2) 
if $I$ and $J$ are related by a certain equivalence relation \cite[p. 103]{shi:book}, then they lie in the same left cell.

Assume Shi's first result above.  Then Fang showed  \cite[Theorem 4.3]{fang:equivalence} that if $I$ and $J$ are related by the basic move, 
then $R_I$ and $R_J$ lie in the same left cell.  
Hence Theorem \ref{thmA} implies

\begin{cor}
	In type $A_n$, the regions $R_I$ and $R_J$ lie in the same left cell whenever $\orbit_I = \orbit_J$.  Hence, Shi's equivalence relation 
	for the positive sign types coincides with the one generated by the basic move.
\end{cor}

In fact, the partition $\lambda_I$ is the same as the partition attached by Lusztig to the two-sided cell that contains $R_I$ \cite{Lusztig:square}.
So the corollary is a variant of the intepretation of the partition in terms of nilpotent orbits for the loop group by Lawton \cite{lawton}.
Fang's result is based on Lusztig's star action, which is analogous to Knuth's relation on $S_n$. 


\section{Combinatorics}
\label{combinatorics}

\subsection{Gerstenhaber's algorithm}\label{algorithm}

For a partition $\mu$ of $n+1$, let $\orbit_\mu$ 
denote the conjugacy class of nilpotent matrices with Jordan blocks of sizes equal to the parts of $\mu$.
If $x \in \orbit_\mu$, we write $\lambda(x)$ for $\mu$.  Define $\lambda(\orbit_I):= \lambda(x)$ for any $x \in \orbit_I$.
Recall that $\orbit_\mu \geq \orbit_{\nu}$ if and only if $\mu \geq \nu$ in the dominance order on partitions.

Given $I \in \ideals$, there is an algorithm to compute $\lambda(\orbit_I)$ due to Gerstenhaber \cite[p. 535]{gerstenhaber}.
The algorithm produces $k$ disjoint, ordered subsets $S_1, \dots, S_k$ of $\{1, 2, \dots, n+1\}$, called the characteristic sequences of $I$.  
Each sequence is defined inductively. The first element of $S_1$ is $1$.   Then if $i \in S_1$, its successor $j$ is defined to be the smallest 
integer such that $t_{ij} \in I$, or equivalently, $\vecdiff{i}{j} \in \ideal.$  
 If no $j$ exists, the sequence terminates.  
Once $S_1, \dots, S_{r}$ are defined, for any matrix we can cross out the rows and columns indexed by the elements in $\cup_{i=1}^r S_i$.  
This defines a map
$p: \mathfrak{gl}_{n+1} \to \mathfrak{gl}_{s}$, where $s = n+1- \# \cup_{i=1}^r S_i$, which takes ideals to ideals.   
Then $S_{r+1}$ is defined as $S_1$ was, but keeping the original labelling of the rows and columns.   
The ideal property of $\ideal$ ensures that $|S_i| \geq |S_{i+1}|$ and Gerstenhaber's partition of $n+1$ is 
defined as $\lambda_I:= (|S_1|, |S_2],\dots)$. 
For each characteristic sequence $S_j = \{ i_1, \dots, i_r\}$, let 
$x_j = t_{i_1, i_2} + t_{i_2, i_3} +\dots + t_{i_{r-1}, i_{r}}$. 
Then the nilpotent matrix 
\begin{equation}
\label{element}
x := \sum_{j=1}^k  x_j \in I, 
\end{equation}
and clearly $\lambda_I = \lambda(x)$ by construction.  
Gerstenhaber proved \cite[Theorem 1]{gerstenhaber} that $x \in \orbit_I$, showing that  
$\lambda(\orbit_I)$ coincides with $\lambda_I$.  We can give a new proof of this result using Theorem \ref{thmA}.

\begin{prop}\label{prop:invariance}
The partition $\lambda_I$ is unchanged under basic moves.
\end{prop}

\begin{proof}
	Suppose $I$ and $J$ are related by the basic move, where the weight space of $e_i-e_j$ 
	is dropped from $I$ to yield $J$. 
	We need to show $\lambda_I = \lambda_{J} $.
	This is clear if $e_i-e_j$ is not selected in any step of the algorithm producing  $\lambda_I$.
	If it is selected, i.e., $j$ follows $i$ in some characteristic sequence 
	for $I$, then there are two cases:  when $I$ is $(i-1)$-stable and when $I$ is $j$-stable.   

		Case 1:  $\ideal$ is stable under the action of $s_{i-1} = (i-1,i)$.  In particular $i\geq2$.  
		We will show that actually $j$ cannot follow $i$ in any characteristic sequence. 
		First, $\vecdiff{i-1}{b} \not\in \ideal$ for any $i-1<b<j$.  For $b=i$, this is clear since $\ideal$
		is $s_{i-1}$ stable.  For $i<b<j$, we would get $s_{i-1}(\vecdiff{i-1}{b} e)) = \vecdiff{i}{b} \in \ideal$.  
		But $\vecdiff{i}{j}$ is a minimal root of $\ideal$, so this cannot happen.  
		Second, $i-1$ cannot be selected in an earlier 
		characteristic sequence than $i$ for $I$.  
		Otherwise, $j$ would be available to follow $i-1$ at the earlier step since $\vecdiff{i-1}{j} \in \ideal$; hence $j$ will be chosen since it is the minimal such value by the first step.  This contradicts that $j$ follows $i$ in a sequence.   
		Next, $i$ cannot begin a characteristic sequence since $i-1$ would be available to start the sequence by the second step, 
		and thus it must be chosen as the first element, ahead of $i$.   
		Finally, if $i$ is part of a sequence that includes $a,i,j$, then $a \neq i-1$ as before (by $s_{i-1}$ stability of $\ideal$).  
		While if $a<i-1$, then $\vecdiff{a}{i} \in \ideal$ implies 
		$s_{i-1}(\vecdiff{a}{i}) = \vecdiff{a}{i-1}  \in \ideal$.  
		Hence, $i-1$ would have followed $a$ since $i-1$ was available to be chosen (intead of $i$). 
		Thus, stability of $\ideal$ under $s_i$ and $\vecdiff{i}{j}$ minimal in $\ideal$ means $j$ cannot follow $i$ in any sequence. 
		We conclude that the characteristic sequences for $I$ and $J$ are identical.

    Case 2: $\ideal$ is stable under $s_j = (j, j+1)$.  
    Let $j$ follow $i$ in the $r$-th sequence for $I$.  
    Suppose $j+1$ appears in an earlier sequence for $I$.  First, 
    it cannot start that sequence since $i<j+1$ and $i$ being available means 
    $i$ or a smaller number would be selected to start the sequence.
    
    Next, if $j+1$ follows some $a$, then $\vecdiff{a}{j+1} \in \ideal$ and then 
    also $\vecdiff{a}{j} \in \ideal$ by $s_j$-stability 
    ($a \neq j$ since $j$ occurs in a later sequence).  
    But $\vecdiff{i}{j}$ is minimal in $\ideal$, so $a<i$. 
    But then $j$ would have followed $a$ in this earlier sequence, instead of $j+1$, a contradiction.  
    We conclude that $j+1$ must appear in a later sequence for $I$ (it cannot occur in the same sequence as $j$
     since $\vecdiff{j}{j+1} \not\in \ideal$ by $s_j$-stability).
    
    Now, since $\vecdiff{i}{j}$ is omitted from $\ideal$ and $j+1$ is still unchosen at the $r$-th step, 
    the $r$-th sequence for $J$  must have $j+1$ following $i$ (note that $\vecdiff{i}{j+1} \in \ideal$ and hence also in $\idealJ$).
  	If some number, say $b$, follows $j$ in the algorithm for $I$, then $\vecdiff{j}{b} \in \ideal$.  And then 
  	$s_j$-stability means $\vecdiff{j+1}{b}$  is in both $\ideal$ and $\idealJ$.  Hence, $s$ follows $j+1$ in the $r$-th sequence for $J$.   
  	This shows that the $r$-th part for $\lambda_I$ and $\lambda_J$ are the same. Also the earlier sequences are identical.

The remainder of $\lambda_I$ and $\lambda_J$ are determined by ideals 
in $ \mathfrak{gl}_{n+1-s}$ under $p$.  
We claim they are the same ideal and that the labeling is the same except that the $j$ label for $p(J)$
becomes the $j+1$ label for $p(I)$.   Indeed, $\vecdiff{a}{b} \in p(J)$ means $\vecdiff{a}{b} \in J$ and $a$ and $b$ are not equal to $j+1$. 
Hence, the indices in $s_j(\vecdiff{a}{b})$ avoid $j$.  Since $I$ is $s_j$-stable, that means $s_j(\vecdiff{a}{b}) \in I$ and then also in $p(I)$.
The converse is similar, concluding the proof that $\lambda_I  = \lambda_J$ if $I \sim J$.
\end{proof}

\begin{cor}[Theorem 1 in \cite{gerstenhaber}]
\label{cor:parabolic_case}
 We have $\lambda(\orbit_I) = \lambda_I$.
\end{cor}
 
 \begin{proof}
By Proposition \ref{thmA} and Lemma \ref{associated_parabolics} and the previous proposition, 
this reduces to computing 
both $\lambda_I$ and $\lambda(\orbit_I)$ 
for $I:= \nilrad_P$ where 
$P = P_{j_1, j_2, \dots, j_l}$ is 
such that the composition $c_i$ in \eqref{composition} is already a partition. 
Let $\mu:= (j_1, j_2-j_1, \dots, j_l - j_{l-1}, n+1- j_l)$ be this partition.

Recall the dual partition $\mu^*$ of $\mu$ is defined by $\mu^*_j = \#\{ i \, |  \, \mu_i \geq j\}$.
So $\mu^*_1 = l+1$.
It is easy to see that $\lambda_I = \mu^*$  (as shown in \cite[Proposition 14]{gerstenhaber}). 
Namtely, set $j_0=0$.   Then the first characteristic sequence is 
$(j_0+1, j_1+1, j_2+1, \dots, j_l+1)$, 
the second is $(j_0+2, j_1+2, \dots, j_{\mu^*_2-1}+2)$, and so on.
The $i$-th sequence has length $\mu^*_i$, so that $\lambda_I = \mu^*$.

Next, the matrix $x \in I$ in \eqref{element} has
$\lambda(x) = \lambda_I$ by construction.  
Hence $x \in \orbit_{\mu^*}$.
The dimension of $\orbit_{\mu^*}$ is $(n+1)^2- \sum \mu_i^2$ \cite{carter}.  
Then by Richardson, $\dim(\orbit_{I}) = 2 \dim (\nilrad_P)$ \cite{carter}
and the latter equals
$\dim GL_{n+1} - \dim (L)$ where $L \simeq \prod GL_{\mu_i}$.  
Hence, $\dim(\orbit_{I}) = \dim(\orbit_{\mu^*})$,  
so it must be that $\orbit_{I} = \orbit_{\mu^*}$.  Therefore,
$\lambda(\orbit_I) = \mu^* = \lambda_I$.
 \end{proof}

\subsection{Two coarser equivalence relations}

First, we recall the bijections among $\ideals$, lattice paths, Dyck paths, and ballot sequences.   
Let $\{ b_i \}_{ i=1}^{2n}$ be a sequence consisting of zeros and ones.
The height $$h_j:=\sum_{i=1}^j  (-1)^{b_i+1} $$
 of the sequence at index $j$ is the number of $1$'s minus the number of $0$'s  
in the subsequence $\{ b_i \}_{ i=1}^{j}$.
A binary sequence $\{ b_i\}_{ i=1}^{2n}$ is called a ballot sequence of length $2n$ if there are
$n$ zeros and $n$ ones and $h_j \geq 0$ for all $j$.
The {\it maximum height} of the ballot sequence is the largest value of $h_j$.

The ballot sequences of length $2n$ give rise 
to a lattice path in the plane as follows:  starting at $(0,n)$ the path moves a unit step east at time $i$ if $b_i=1$ 
and a unit step south if $b_i=0$. Since there are $n$ ones and $n$ zeros, the path terminates at $(n,0)$.  Since $h_j\geq 0$ for all $j$, the path stays
at or above the line joining $(0,n)$ and $(n,0)$.  
This is a bijection between ballot sequences and such lattice paths.  Such
lattice paths are then in bijection with the ideals in $\mathfrak{gl}_{n}$.  The corresponding ideal $I$ is the one whose border is this 
lattice path where the lower left corner of a matrix is at $(0,0)$; namely, where $t_{ij} \in I$ if and only if the point $(n-i,j)$ 
lies on or northeast of the lattice path.  Finally, these lattice paths are clearly in bijection with the Dyck paths in the plane, 
which start at $(0,0)$ and move by $(1,1)$ is $b_i=1$ and by $(1,-1)$ if $b_i=0$ and stay above the $x$-axis.  Then the notion of maximum height 
is just the $y$-coordinate of the largest peak in the Dyck path.


The first characteristic sequence $(i_1= 1, i_2, \dots, i_{r})$ of $I$ 
was subsequently re-introduced in the literature \cite{andrews-et-al:nilpotence} and then 
adapted to Dyck paths \cite{haglund}, where it determines the {\it bounce path}.
Visually, this is the lattice path that traces the lower 
boundary of the ideal $I:=\nilrad_P$ where $P = P_{i_2-1, \dots, i_{r}-1}$.  
The value $(\lambda_I)_1-1 = r-1$ is called the {\it bounce count} or bounce number of the path 
since this is the number of times the path touches the diagonal (strictly inside
the matrix); namely, it touches at the points $(i_2-1, i_2-1), \dots, (i_r-1, i_r-1)$.


Next, we break the basic move down into two coarser moves.
\begin{defn} 
	Let $[i,j-1]=\vecdiff{i}{j}$ be a minimal root of $\ideal$.   
	
	If either 
	\begin{enumerate}
	\item $i \geq 2$ and $e_{a}-e_i$ is not a minimal root for all $a<i$,  or 
	\item $j\leq n$ and $e_{j} -e_b$  is not a minimal root for all $b>j$,
	\end{enumerate}
	then we say $\ideal$ and $\ideal \backslash \{ \vecdiff{i}{j} \}$ are equivalent under the {\it inner move}.
	
	
	If either 
	\begin{enumerate}
	\item $i \geq 2$ and $e_{i-1}-e_a$ is not a minimal root for all $a>i-1$,  or
	\item $j\leq n$ and $e_{b} -e_{j+1}$  is not a minimal root for all $b<j+1$,
	\end{enumerate}
	then we say $\ideal$ and $\ideal \backslash \{ \vecdiff{i}{j} \}$ are equivalent under the {\it outer move}.

\end{defn}

 \begin{prop}
 	If $I$ and $J$ are equivalent under the inner move, then $(\lambda_I)_1 = (\lambda_J)_1$.  
	In the language of Dyck paths, the bounce count is invariant under the inner move.
\end{prop}

\begin{proof}
	Assume we are omitting  $e_i-e_j$ from $\ideal$ by the inner move.
	We need only be concerned if $e_i-e_j$ is selected in the constructing the first characteristic sequence for $I$.  
	
	If the first condition of the inner move applies, then $i \geq 2$ so $i$ does not begin the sequence.   
	If $(a, i, j)$ appears in the first sequence, then $e_a-e_i \in \ideal$.   But this root cannot be minimal, so let $e_b - e_c$
	be a minimal root below it in the partial order, i.e., $b\geq a$ and $c\leq i$.  
	If $c<i$, then $e_a-e_{i-1}\in \ideal$, so $i-1$ would be chosen to follow $a$, instead of $i$, in the algorithm since $i-1<i$.   Hence, $c=i$, which contradicts the condition on minimal roots.  
	We conclude $j$ cannot follow $i$ in the first characteristic sequence. 
	
	If the second condition of the inner move applies, the proof is same as part of Case 2 in the Proposition \ref{prop:invariance}.
	First, $j+1$ cannot be selected in the first sequence; otherwise, $\vecdiff{j}{j+1} \in I$, hence minimal, which is not allowed.
	Thus, $j+1$ follow $i$ in the first sequence of $J$.  Finally, if $b$ follows $j$ in the first sequence, then as in the proof of the proposition,
	$b$ follows $j+1$ in the first sequence of $J$ since the argument only used that $\vecdiff{j}{b} \not\in \ideal_{min}$.
\end{proof}

\begin{prop} \label{max_height}
	The number of parts of $\lambda_I$ is the maximum height 
	of the associated ballot sequence (or Dyck path).
\end{prop}

\begin{proof}
	First, we show that the outer move preserves the maximum height of a ballot sequence.  
	In the language of ballot sequences,  the minimal roots of $\ideal$ correspond to spots with $b_r = 0$ and $b_{r+1} =1$.  
Set $j = \sum_{k \leq r} b_k$, the number of $1$'s up to and including the $1$ in this $01$ subsequence.
Set $i = r-j$, the number of $0$'s up to and including this $01$ subsequence.  
Then $\vecdiff{i}{j}$ is a minimal root of $\ideal$.  
The first condition of the outer move means that any $01$ subsequence extends on its left to $001$, that is, $b_{r-1}=0$. 
The second condition of the outer move says that the subsequence extends on its right to $011$, that is, $b_{r+2}=1$.
Now, dropping this minimal root from the ideal 
changes the ballot sequence by replacing $01$ with $10$.  If either the first or the second condition
of the outer move hold, so that $001$ changes to $010$ or $011$ changes to $101$, the maximum height of the sequence is clearly unchanged.  

Next, $I=\nilrad_P$ as in Corollary \ref{cor:parabolic_case}
has maximum height equal to $\mu_1 = (\lambda_I)^*_1$, i.e., the number of parts of $\lambda_I$.
By Theorem \ref{thmA} any ideal is equivalent to some $\nilrad_P$ by a sequence of basic moves, hence the result follows 
since outer moves (hence basic moves) preserve maximum height and 
basic moves preserve $\lambda_I$ by Proposition \ref{prop:invariance}.
	\end{proof}

\begin{rmk}
Both bounce count and maximum height have an algebraic interpretation.
Since $\lambda(\orbit_I) = \lambda_I$ by Corollary \ref{cor:parabolic_case}, 
it is clear  that $(\lambda_I)_1$ is the smallest positive integer $k$ such $A^{k}=0$ for all $A \in I$.
Moreover, it is pointed out in \cite[p. 536]{gerstenhaber} that $(\lambda_I)_1$ 
is the index of nilpotence of $I$ as an {\bf associative algebra}, while in \cite{andrews-et-al:nilpotence} it is shown 
that $(\lambda_I)_1$ is the index of nilpotence of $I$ as a {\bf Lie algebra} 
(actually, the latter use the class of nilpotence to refer to the number of brackets needed to get to zero, 
which equals $(\lambda_I)_1-1$, i.e., the bounce count).
We will say that $I$ has index $k$.

On the other hand, the number of parts of $\lambda_I$ is the smallest $k$ such that $A$ has rank at most $n-k$ for all $A \in I$. We say that $I$ has corank $k$.
\end{rmk}  


In \cite[Section 5]{andrews-et-al:nilpotence} a bijection $\iota: \ideals \to \ideals$ was introduced, where it was shown that 
$\iota$ sends an ideal of bounce count $k-1$ to one of maximum height $k$.   
This has a nice interpretation in terms of Proposition \ref{max_height}.  Namely, for each $1 \leq k \leq n+1$,
\begin{equation}\label{index_corank1}
\iota(\{ I \in \ideals  \, | \,  (\lambda_I)_1 = k \}) = \{ I \in \ideals  \, | \,  (\lambda^*_I)_1 = k \},
\end{equation}
so that $\iota$ sends an ideal of index $k$ to one of corank $k$.
The map $\iota$ also shows up as the inverse to the sweep map on Dyck paths \cite{haglund}.

\begin{rmk}  
We wonder whether 
the inner moves generates the equivalence class of ideals of index $k$ (i.e., Dyck paths with bounce count $k-1$), and
whether the outer move generates the equivalence class of ideals of corank $k$ (i.e., Dyck paths with maximum height $k$).
\end{rmk}

\subsection{Unit Interval Orders}

Another incarnation of $\ideals$ in type $A_n$ is via unit interval orders.  
These are certain partial orders on the set 
$\{1, 2, \dots, n+1\}$.  
Given $I \in \ideal$, 
the corresponding unit interval order $P_I$ is defined by $i \prec j$ if and only if $t_{ij} \in I$.   

Consider a set partition of $P_I$ into a union of $k$ disjoint chains $C_1\cup \dots \cup C_k$
with $\#C_i \geq \#C_{i+1}$.  Let $\mu$ be the partition with $\mu_i = \# C_i$.
Then we can define a nilpotent element $x \in I$ as in \eqref{element} and $x$ satisfies $\lambda(x) = \mu$. 
Hence $\lambda_I \geq \mu$ since $\lambda_I = \lambda(x_I) \geq \lambda(x)$ by Corollary \ref{cor:parabolic_case}.
On the other hand, $\cup_{i=1}^k S_i$, where $S_i$ are as in \S\ref{algorithm}, is a disjoint union of $k$ chains in $P_I$. 
Therefore $\sum_{i=1}^k (\lambda_I)_i$ is the maximal cardinality of any such decomposition of $P_I$.  
Thus 
(see \cite{britz-fomin})
\begin{prop}
$\lambda_I$ is the Greene-Kleitman partition attached to $P_I$.
\end{prop}

There is another perspective from the point of view of the indifference graph of $P_I$ (see \cite{brosnan_chow}).  This is the 
undirected graph $\bf G_I$ on $\{1, 2, \dots, n+1\}$ 
where there is an edge between $i$ and $j$ if and only if $i$ and $j$ are not comparable in $P_I$. 
Clearly any decomposition for $P_I$ in $k$ disjoint chains
corresponds to a decomposition of $\bf G_I$ into a disjoint union of $k$ independent sets, and vice versa.   
The vertices in each independent set
can be colored the same color.  This means any coloring of $\bf G_I$ gives rise to a chain decomposition, and vice versa.
In particular, the largest part of $\lambda_I$ gains another interpretation as the independence number of $\bf G_I$
and the number of parts of $\lambda_I$ is the chromatic number of $\bf G_I$.
Tim Chow (private communication) 
proposed a coloring algorithm for $\bf G_I$ that is equivalent to  the one in \S\ref{algorithm} and showed, without reference to nilpotent orbits, that 
$\lambda_I$ dominates the partition associated to any coloring of $\bf G_I$.  

Finally, we point out that the first numbers of the $S_i$  in \S\ref{algorithm} form an antichain in $P_I$, which implies 
the number of parts of $\lambda_I$ is the size of any maximal antichain in $P_I$ 
and therefore equals the clique number of $\bf G_I$ (which then also equals the chromatic number).

\subsection{Enumeration}

Several results related to enumerating the ideals of index $k$, hence also those of corank $k$, 
are given in \cite{andrews-et-al:nilpotence}.   
It would be nice to be able to find a formula for the cardinality 
$N_\lambda$ of the equivalence class attached to $\orbit_\lambda$.  
We the list the values of 
$N_\lambda$ for low ranks in Table \ref{tab:N}.   
For $n \leq 4$, we have $N_\lambda = N_{\lambda^*}$, 
but the numbers start to diverge for larger $n$.
Still, we know from  \eqref{index_corank1} that for all $k$
\begin{equation}\label{index_corank}
\sum_{\lambda: \lambda_1 = k} N_\lambda = \sum_{\lambda: \lambda_1 = k} N_{\lambda^*},
\end{equation}
so there may be some further connection between $N_\lambda$ and $N_{\lambda^*}$.
From the last section, Equation \eqref{index_corank} can also be phrased in terms 
of indifference graphs on unit interval orders, using independence number and clique number in place of bounce count and maximum height.

%
%
%
%





\begin{table}
\begin{center}
		\begin{tabular}{|l|c||l|c||l|c||l|c|} \hline
		
			\multicolumn{2}{|c||}{$A_6$}
			& \multicolumn{2}{|c||}{$A_5$}
			& \multicolumn{2}{|c||}{$A_4$}  
		        & \multicolumn{2}{|c||}{$A_3$} \\ \hline
			\multicolumn{1}{|c|}{$\lambda$} 
			& \multicolumn{1}{|c||}{$N_\lambda$}  
			
			&\multicolumn{1}{|c|}{$\lambda$} 
			& \multicolumn{1}{|c||}{$N_\lambda$} 
			
			&\multicolumn{1}{|c|}{$\lambda$} 
			& \multicolumn{1}{|c||}{$N_\lambda$} 
			
			&\multicolumn{1}{|c|}{$\lambda$} 
			& \multicolumn{1}{|c|}{$N_\lambda$} 
			
			\\ \hline 
			
			$[7]$  & 1 &&&&&& \\
			$[6,1]$  & 11  &&&&&& \\
			$[5,2]$  & 32 &&&&&& \\  
			$[5,1^2]$  & 20 &&&&&& \\
			$[4,3]$  & 20   & $[6]$  & 1  &&&& \\
			$[4,2,1]$  &  87 & $[5,1]$  & 9 &&&& \\
			$[4,1^3]$ & 25 & $[4,2]$  &  18 &&&& \\ 
			$[3^2,1]$  & 29 & $[4,1^2]$ & 15 &&&& \\
			$[3,2^2]$  & 33 &  $[3^2]$  & 4 &  $[5]$  & 1 &&\\ 
			$[3,2,1^2]$  & 84 & $[3,2,1]$  & 37 &  $[4,1]$ & 7 &&\\ 
			$[3,1^4]$ & 23 & $[3,1^3]$ & 16 & $[3,2]$  & 8 &  $[4]$ & 1 \\  
			$[2^3,1]$  &  23 & $[2^3]$  &  5 &  $[3,1^2]$ & 10 & $[3,1]$ & 5 \\ 
			$[2^2,1^3]$ &  29 &$[2^2,1^2]$ &  17 &$[2^2,1]$ & 8 &$[2^2]$ & 2  \\
			$[2,1^5]$ & 11 &  $[2,1^4]$ & 9 &   $[2,1^3]$ & 7 & $[2,1^2]$ & 5 \\  
			$[1^7]$ & 1 &     $[1^6]$ & 1 &   $[1^5]$ & 1  &  $[1^4]$ & 1\\ 

			\hline 
		\end{tabular}
		\caption{\label{tab:N}Values of $N_\lambda$.}
\end{center}

\end{table}


There is a second partition attached to an ideal coming from the minimal roots in $\ideal$.   Namely, if we define
$$e_I :=\sum_{\beta \in \ideal_{min}}  e_\beta,$$
then we can also associate to $I$ the partition $\lambda(e_I)$, which we call the Kreweras partition of $I$.  
The element
$e_I$ can be defined for any Lie type.  
Since $\ideal_{min}$ are the simple roots of the parabolic subsystem of $\ro$ they span \cite[Theorem 1]{sommers:ideals_combo},  
the element $e_I$ is always regular in a Levi subalgebra (outside of type $A$, the elements in $\orbit_I$ are not always 
regular in a Levi subalgebra).   
Since $e_I \in I$, in type $A_n$ it follows that 
\begin{equation}\label{dominance}
\lambda_I \geq \lambda(e_I). 
\end{equation}

Unlike $N_{\lambda}$, there are closed formulas for $K_\lambda := \# \{ I \in \ideals \, | \,  \lambda(e_I) = \lambda \}$. 
 In fact, the formulas exist in all Lie types, where they involve the exponents of a hyperplane arrangement attached to the parabolic root system determined by $\ideal_{min}$ \cite[Proposition 6.6]{sommers:ideals_combo}.  In type $A_n$, they coincide with the Kreweras numbers, which are given in terms of multinomials by
$$ K_\lambda = \frac{1}{n+2} \binom{n+2}{ a_1(\lambda),a_2(\lambda), \ldots, a_{n+1}(\lambda), n+2- \ell},$$
where $a_j(\lambda)$ denotes the multiplicity of the number
 $j$ among the parts of $\lambda$ and $\ell$ is the total number of parts. 
 

Let $m_I$ denote the number of minimal roots of $I$.  Then $m_I$ equals the rank of $e_I$, which is $(n+1) - \lambda(e_I)^*_1$. 
In terms of Dyck paths, $m_I$ is the number of valleys of the Dyck path, while $m_I+1$ is the number of peaks.
Then $\# \{ I \in \ideals \, | \,  m_I =k  \}$ give the Narayana numbers (\cite[Proposition 4.1]{panyushev:duality}).
We make the following observation based on analyzing the construction of $\iota$ in \cite{andrews-et-al:nilpotence}.
\begin{prop} \label{narayana_inversion}
With respect to $\iota$, we have 
\begin{equation}  \label{narayana}
m_I = n - m_{\iota(I)}
\end{equation}   
for $I \in \ideals$.
Putting \eqref{narayana} together with \eqref{index_corank1} we get
\begin{equation}  \label{bistat_narayan}
\iota(\{ I \in \ideals  \, | \,  (\lambda_I)_1 = r, m_I = s \}) = \{ I \in \ideals  \, | \,  (\lambda^*_I)_1 = r, m_I = n-s \}
\end{equation}
for all $r,s$.
\end{prop} 
Panyushev also constructed a bijection that satisfies equation \eqref{narayana} in \cite{panyushev:duality}.  That bijection is an involution,
whereas $\iota$ can have large order.

\begin{table} 
	\begin{center}
		\begin{tabular}{|l||l|c|c|c|c|c|c|c|c|c|c||c|} \hline
			$(\lambda_I)_1, m_I$  & 10 & 9& 8 & 7 &6&5&4&3&2&1&0 &  \\ \hline		
			$(\lambda_I)^*_1, m_I$   & 0 & 1& 2 & 3 & 4&5&6&7&8&9&10 &  \\ \hline			
			\hline
			1 & 	0 & 0 & 0 & 0 & 0 & 0 & 0 & 0 & 0 & 0 & 1 & 1\\
			2 & 	0 & 0 & 0 & 0 & 0 & 11 & 165 & 462 & 330 & 55 & 0  & 1023\\ 
			3  & 	0 & 0 & 0 & 7 & 301 & 2090 & 4257 & 2772 & 495 & 0 & 0 & 9922   \\
			4  & 	0 & 0 & 5 & 309 & 2821 & 7293 & 6435 & 1716 & 0 & 0 & 0 & 18579  \\ 
			5  & 	0 & 0 & 65 & 1119 & 4823 & 7007 & 3003 & 0 & 0 & 0 & 0  & 16017 \\ 
			6 & 	0 & 3 & 162 & 1515 & 4095 & 3003 & 0 & 0 & 0 & 0 & 0 & 8778 \\
			7  & 	0 & 7 & 219 & 1320 & 1820 & 0 & 0 & 0 & 0 & 0 & 0 & 3366    \\ 
			8  & 	0 & 11 & 221 & 680 & 0 & 0 & 0 & 0 & 0 & 0 & 0 & 912 \\
			9  & 	0 & 15 & 153 & 0 & 0 & 0 & 0 & 0 & 0 & 0 & 0& 168 \\ 
			10& 	0 & 19 & 0 & 0 & 0 & 0 & 0 & 0 & 0 & 0 & 0 & 19 \\
			11 & 1& 0 & 0 & 0 & 0 & 0 & 0 & 0 & 0 & 0 & 0 & 1\\			
			\hline
			&	1  & 55  & 825 & 4950 & 13860 & 19404 & 13860 & 4950 & 825 & 55 & 1 & sum\\		\hline
			
		\end{tabular}
		\caption{\label{bistatistic} Joint valley-bounce count (or valley-maximum height) statistics for $A_{10}$  }
		
	\end{center}
\end{table}

We list the cardinality of the sets in \eqref{bistat_narayan} for $A_{10}$ in Table \ref{bistatistic}.
This matrix for general $n$ has zeros below the anti-diagonal since 
\eqref{dominance} implies $(\lambda^*_I)_1 \leq (\lambda(e_I)^*)_1$,
hence $m_I \geq (n+1) - (\lambda^*_I)_1$. 
With the indexing from the first row, 
the anti-diagonal entries seem to be $\binom{n+m_I}{n-m_I}$, see A054142 from \cite{oeis}. 
Finally, $m_I \leq n+1-\lceil \frac{n+1}{(\lambda_I)_1} \rceil$ since this is the larger possible rank of a matrix
given its index of nilpotence. This explains the zeros in the upper left corner. 
It would be interesting to find formulas for the other cardinalities of the sets in \eqref{bistat_narayan},
or better yet, the joint statistic on the two partitions  $\lambda_I$ and $\lambda(e_I)$
attached to $I$.

\section*{Acknowledgments}
We thank Vic Reiner for suggesting in a 2007 email that a statement like Proposition \ref{narayana_inversion} should hold.  We thank
Bill Casselman for directing us to Gerstenhaber's algorithm and additional helpful comments.  We thank  Tim Chow, Martha Precup, and John Shareshian
for helpful conversations related to  the unit interval order.

\bibliography{A_ideal_equivalence_new}
\bibliographystyle{pnaplain}

\end{document}